\documentclass[oneside,english]{amsart}
\usepackage[T1]{fontenc}
\usepackage[latin9]{inputenc}
\usepackage{amstext}
\usepackage{amsthm}
\usepackage{amssymb}
\usepackage{cancel}

\makeatletter
\numberwithin{equation}{section}
\numberwithin{figure}{section}
\theoremstyle{plain}
\newtheorem{thm}{\protect\theoremname}
  \theoremstyle{remark}
  \newtheorem*{rem*}{\protect\remarkname}
  \newtheorem{rem}{\protect\remarkname}
  \theoremstyle{plain}
  \newtheorem{cor}[thm]{\protect\corollaryname}

\DeclareMathOperator{\sech}{sech}

\makeatother

\usepackage{babel}
  \providecommand{\corollaryname}{Corollary}
  \providecommand{\remarkname}{Remark}
\providecommand{\theoremname}{Theorem}

\begin{document}

\title[Higher-order Bernoulli and Euler Polynomials: Random Walk]{Connection Coefficients for Higher-order Bernoulli and Euler Polynomials: A Random Walk Approach}

\author{Lin Jiu}
\address{Department of Mathematics and Statistics\\
                Dalhousie University\\
                Halifax, NS\\
                B3H 4R2, Canada}
\email{Lin.Jiu@dal.ca}

\author{Christophe Vignat}
\address{Department of Mathematics\\
               Tulane University\\
               New Orleans, LA\\
               70118, USA}
\email{cvignat@tulane.edu}

\begin{abstract}
We consider the use of random walks as an approach to obtain connection coefficients for higher-order Bernoulli and Euler polynomials. In particular, we
consider the cases of a $1$-dimensional linear reflected Brownian
motion and of a $3$-dimensional Bessel process. Considering the successive
hitting times of two, three, and four fixed levels by these random walks
yields non-trivial identities that involve higher-order Bernoulli and
Euler polynomials.
\end{abstract}

\maketitle

\section{Introduction}

Different types of random walks have been studied in the literature,
together with their connections to different fields of mathematics
and physics; for a modern introduction, see for example \cite{Lawler}.
In this paper, we focus on two specific random walks: the $1$-dimensional
linear reflected Brownian motion and the $3$-dimensional Bessel process.
It would seem that there is no relation between these processes and special functions such as Bernoulli and Euler polynomials that appear mostly in number theory and combinatorics. However,
we will show how the study of the hitting times of these two processes
allows us to derive non-elementary identities for higher-order Bernoulli
and Euler polynomials, denoted by $B_{n}^{\left(p\right)}\left(x\right)$
and $E_{n}^{\left(p\right)}\left(x\right)$, respectively. 

These polynomials are defined through their generating functions as
\begin{equation}
\left(\frac{t}{e^{t}-1}\right)^{p}e^{xt}=\sum_{n=0}^{\infty}B_{n}^{\left(p\right)}\left(x\right)\frac{t^{n}}{n!}\ \ \ \text{and}\ \ \ \left(\frac{2}{e^{t}+1}\right)^{p}e^{xt}=\sum_{n=0}^{\infty}E_{n}^{\left(p\right)}\left(x\right)\frac{t^{n}}{n!}.\label{eq:GF}
\end{equation}
The special case $p=1$ yields the usual Bernoulli and Euler polynomials:
$B_{n}^{\left(1\right)}\left(x\right)=B_{n}\left(x\right)$ and $E_{n}^{\left(1\right)}\left(x\right)=E_{n}\left(x\right)$;
in addition, Bernoulli and Euler numbers are the special evaluations
$B_{n}=B_{n}\left(0\right)$ and $E_{n}=2^{n}E_{n}\left(1/2\right)$. 

One of the key tools of this article is the interpretation
of the polynomials $B_{n}^{\left(p\right)}\left(x\right)$ and $E_{n}^{\left(p\right)}\left(x\right)$
as probabilistic moments of certain random variables that are related to the hitting times of some random walks,
which will be introduced in Section \ref{sec:Pre}. 

This study arises from early work on the higher-order Euler polynomials. In a previous article \cite[eq.~3.8, p.~781]{Euler},
we obtained the following expansion for the usual Euler polynomials
as a linear combination of higher-order Euler polynomials: for any
positive integer $N$, 
\begin{equation}
E_{n}(x)=\frac{1}{N^{n}}\sum_{l=N}^{\infty}p_{l}^{(N)}E_{n}^{(l)}\left(\frac{l-N}{2}+Nx\right),\label{eq:Euler}
\end{equation}
where the positive coefficients $p_{l}^{(N)}$ are defined by the
generating function 
\[
\frac{1}{T_{N}(1/t)}=\sum_{l=0}^{\infty}p_{l}^{(N)}t^{l},
\]
via the $N$-th Chebychev polynomial $T_{N}(t)$. These coefficients $p_{l}^{(N)}$ also appear as transition probabilities
in the context of a random walk over a finite number of sites \cite[Note 4.8, p.~787]{Euler}.
This characterization provides an alternate interpretation of \eqref{eq:Euler}
in  terms of a Markov process.

For the two random walks considered in this paper, we shall
focus on their consecutive hitting times: namely, for a process starting
from the origin at time $t=0$, we consider the first epochs - called hitting times - at which the process reaches successive
levels $0=a_{0}<a_{1}<a_{2}<\cdots<a_{N}$. In particular, we study
some special cases for $N=2$, $3$, and $4$. From the moment generating
functions of these hitting times, we derive identities involving $B_{n}^{\left(p\right)}\left(x\right)$
and $E_{n}^{\left(p\right)}\left(x\right)$ which will be presented
in Section \ref{sec:OneDim} and Section \ref{sec:ThreeDim}. 

\section{\label{sec:Pre}Probabilistic Preliminaries and Principle }

\subsection{The symbols $\mathcal{B}$, $\mathcal{E}$ and $\mathcal{U}$}

We shall fully make use of the classical umbral symbols $\mathcal{B}$,
$\mathcal{E}$, and $\mathcal{U}$, defined as follows - for an introduction
to the classical umbral calculus, see for example \cite{Gessel,umbral}. The Bernoulli
symbol $\mathcal{B}$ satisfies the evaluation rule 
\begin{equation}
\left(\mathcal{B}+x\right)^{n}=B_{n}\left(x\right).\label{eq:Bernoulli polynomial}
\end{equation}
Equivalently, this symbol $\mathcal{B}$
can be interpreted as a random variable \cite[Thm.~2.3, p.~384]{Zagier1}:
if $L_{B}$ is a random variable distributed according to the square
hyperbolic secant density $p_{B}\left(t\right)=\pi\sech^{2}\left(\pi t\right)/2$, 
then, for any suitable function $f$, 
\[
f\left(x+\mathcal{B}\right)=\mathbb{E}\left[f\left(x+iL_{B}-\frac{1}{2}\right)\right]=\int_{\mathbb{R}}f\left(x+it-\frac{1}{2}\right)p_{B}\left(t\right)\mathrm{d}t.
\]
In particular, choosing $f\left(x\right)=x^{n}$ produces the Bernoulli
polynomial (\ref{eq:Bernoulli polynomial}, so that $\mathcal{B}$ coincides with the  random variable $iL_B -1/2$. More generally, the $p$-th order
Bernoulli polynomial can be expressed as 
\[
B_{n}^{\left(p\right)}\left(x\right)=\left(x+\mathcal{B}_{1}+\cdots+\mathcal{B}_{p}\right)^{n},
\]
where $\left(\mathcal{B}_{i}\right)_{i=1}^{p}$ is a set of $p$ independent
umbral symbols (or random variables) that satisfy the two following
evaluation rules: 

1: if $\mathcal{B}_{1}$ and $\mathcal{B}_{2}$ are independent symbols
then
\[
\mathcal{B}_{1}^{n}\mathcal{B}_{2}^{m}=B_{n}B_{m};
\]

2: if $\mathcal{B}_{1}=\mathcal{B}_{2}$,
\[
\mathcal{B}_{1}^{n}\mathcal{B}_{2}^{m}=\mathcal{B}_{1}^{n+m}=B_{n+m}.
\]
For simplicity, we shall denote the sum of $p$ independent Bernoulli
symbols as $\mathcal{B}^{\left(p\right)}=\mathcal{B}_{1}+\cdots+\mathcal{B}_{p}$. 

From the generating function \eqref{eq:GF}, we deduce that
\[
e^{\mathcal{B}t}=\frac{t}{e^{t}-1},\ \ \ e^{t\left(2\mathcal{B}+1\right)}=\frac{t}{\sinh t},\ \ \ \text{and}\ \ \ e^{t\left(2\mathcal{B}^{\left(p\right)}+p\right)}=\left(\frac{t}{\sinh t}\right)^{p}.
\]
Similarly, we let $L_{E}$ be a random variable distributed according
to the hyperbolic secant density $p_{E}\left(t\right)=\sech\left(\pi t\right)$, 
and define the Euler symbol $\mathcal{E}$ by
\[
f\left(x+\mathcal{E}\right)=\mathbb{E}\left[f\left(x+ iL_{E}-\frac{1}{2}\right)\right]=\int_{\mathbb{R}}f\left(x+ it-\frac{1}{2}\right)p_{E}\left(t\right)\mathrm{d}t.
\]
Then, denoting $\mathcal{E}^{\left(p\right)}=\mathcal{E}_{1}+\cdots+\mathcal{E}_{p}$,
the sum of $p$ independent Euler symbols, the Euler polynomial of
order $p$ is expressed as
\[
E_{n}^{\left(p\right)}\left(x\right)=\left[x+\mathcal{E}^{\left(p\right)}\right]^{n}.
\]
Also, we have
\[
e^{t\mathcal{E}}=\frac{2}{e^{t}+1},\ \ \ e^{t\left(2\mathcal{E}+1\right)}=\sech t ,\ \ \ \text{and}\ \ \ e^{t\left(2\mathcal{E}^{\left(p\right)}+p\right)}=\sech^{p} t .
\]

From the generating function
\[
e^{2\mathcal{B}t}=\frac{2t}{e^{2t}-1}=\frac{t}{e^{t}-1}\cdot\frac{2}{e^{t}+1}=e^{t\left(\mathcal{B}+\mathcal{E}\right)},
\]
we deduce that the symbols $\mathcal{B}$ and $\mathcal{E}$ satisfy the rule
\begin{equation}
2\mathcal{B}=\mathcal{B}+\mathcal{E},\label{eq:2BE}
\end{equation}
in the sense that for any suitable function,
\[
f\left(x+2\mathcal{B}\right)=f\left(x+\mathcal{B}+\mathcal{E}\right).
\]

The last useful symbol is the uniform symbol $\mathcal{U}$ with the
evaluation rule
\[
\mathcal{U}^{n}=\frac{1}{n+1}.
\]
It can be interpreted as a random variable $U$ uniformly distributed
over the interval $\left[0,1\right]$, and for any suitable function
$f$, 
\[
f\left(x+\mathcal{U}\right)=\mathbb{E}\left[f\left(x+U\right)\right]=\int_{0}^{1}f\left(x+t\right)\mathrm{d}t.
\]
Thus, we have in terms of generating functions
\[
e^{t\mathcal{U}}=\sum_{n=0}^{\infty}\frac{t^{n}}{\left(n+1\right)!}=\frac{e^{t}-1}{t},\thinspace\thinspace e^{t\left(2\mathcal{U}-1\right)}=\frac{\sinh t }{t},\thinspace\thinspace\text{and}\thinspace\thinspace e^{t\left(2\mathcal{U}^{\left(p\right)}-p\right)}=\left(\frac{\sinh t }{t}\right)^{p}.
\]
where $\mathcal{U}^{\left(p\right)}=\mathcal{U}_{1}+\cdots+\mathcal{U}_{p}$
denotes the sum of $p$ independent uniform symbols. 

An important link between the Bernoulli symbol $\mathcal{B}$ and
the uniform symbol $\mathbf{\mathcal{U}}$ is deduced from the identity
\[
e^{t\left(\mathcal{U}+\mathcal{B}\right)}=e^{t\mathcal{U}}e^{t\mathcal{B}}=\frac{t}{e^{t}-1}\cdot\frac{e^{t}-1}{t}=1;
\]
this shows that, for any suitable function $f$,  
\begin{equation}
f\left(x+\mathcal{B}+\mathcal{U}\right)=f\left(x\right), \label{eq:cancel}
\end{equation}
so that the actions of these two symbols cancel each other. 

In what follows, we will use independent copies of Bernoulli, Euler
and uniform symbols that, for example in the uniform case, they will be
denoted by $\mathcal{U},\thinspace\mathcal{U}',\dots$ and $\mathcal{U}{}^{\left(p\right)},\mathcal{U}'^{\left(p\right)},\dots$ in order to be distinguished.

\subsection{\label{subsec:loops}Level sites with $1$ loop and $2$ loops }

The general setting of this paper is thus a random walk starting from
the origin (in $\mathbb{R}$ or $\mathbb{R}^{3}$) and hitting some
defined levels -either some points on the line or some sphere in the
3-dimensional space, called sites. 
Looping back and forth, each random
walk can reach each site multiple times, and we are interested in
the first time the process reaches each of these sites. Before we derive identities,
we shall first consider the contribution of loop(s) to the hitting
times of these sites. 

Consider one possible loop between sites $a$ and $b$, with $a<b$.
Let $\phi_{a\rightarrow b}$ be the moment generating function of
the hitting time of site $b$ staring from site $a$, without hitting
any other site if any; and also let $\phi_{b\rightarrow a}$ be the
counterpart from $b$ to $a$ of $\phi_{a\rightarrow b}$. Let us further denote by
\[
I_{a,b}=\phi_{a\rightarrow b}\phi_{b\rightarrow a};
\]
the random walk can loop an arbitrary number of times $k\ge0$ between
sites $a$ and $b$ (without visiting any other sites) so
that the overall contribution of these visits to the generating function
is
\begin{equation}
\sum_{k=0}^{\infty}I_{a,b}^{k}=\frac{1}{1-I_{a,b}}.\label{eq:OneLoop}
\end{equation}

Next, consider one possible loop between $a$ and $b$, and another
between $c$ and $d$, with $a<b\leq c<d$. Using similar notations
as above, we define 
\[
I_{a,b}=\phi_{a\rightarrow b}\phi_{b\rightarrow a}\ \ \ \text{and}\ \ \ I_{c,d}=\phi_{c\rightarrow d}\phi_{d\rightarrow c}.
\]
The possible contributions are as follows:
\begin{itemize}
\item $k$ loops between sites $a$ and $b$ followed by $l$ loops between
sites $c$ and $d$, with $k,l=0,1,\ldots$, contributing
\[
\sum_{k.l\ge0}I_{a,b}^{k}I_{c,d}^{l}=\frac{1}{1-I_{a,b}}\cdot\frac{1}{1-I_{c,d}};
\]
\item $k_{1}$ loops between sites $a$ and $b$ followed by $l_{1}$ loops
between sites $c$ and $d$, then followed by $k_{2}$ loops between
sites $a$ and $b$ and finally $l_{2}$ loops between sites $c$
and $d$, with $k_{1},\ l_{2}$ nonnegative and $k_{2}$, $l_{1}$
positive, contributing
\[
\sum_{k_{1}.l_{2}=0,k_{2},l_{1}=1}^{\infty}I_{a,b}^{k_{1}}I_{c,d}^{l_{1}}I_{a,b}^{k_{2}}I_{c.d}^{l_{2}}=\frac{I_{a,b}I_{c,d}}{\left(1-I_{a,b}\right)^{2}\left(1-I_{c,d}\right)^{2}};
\]
\item the general term will consist of $k_{1}$ loops between sites $a$ and $b$
followed by $l_{1}$ loops between sites $c$ and $d$ and so on followed
by $k_{n}$ loops between sites $a$ and $b$ and $l_{n}$ loops between
sites $c$ and $d$ , with $k_{1}$, $l_{n}$ nonnegative and the
other indices being positive, overall contributing
\[
\frac{I_{a,b}^{n-1}I_{c,d}^{n-1}}{\left(1-I_{a,b}\right)^{n}\left(1-I_{c,d}\right)^{n}}
\]
to the generating function.
\end{itemize}
Therefore, the overall contribution of all possible loops over sites
$a$ and $b$ followed by loops over sites $c$ and $d$ is
\begin{equation}
\sum_{n=1}^{\infty}\frac{I_{a,b}^{n-1}I_{c,d}^{n-1}}{\left(1-I_{a,b}\right)^{n}\left(1-I_{c,d}\right)^{n}}=\frac{1}{1-\left(I_{a,b}+I_{c,d}\right)}=\sum_{k=0}^{\infty}\left(I_{a,b}+I_{c,d}\right)^{k}.\label{eq:TwoLoops}
\end{equation}

\section{\label{sec:OneDim}One-Dimensional Reflected Brownian Motion}

\subsection{Introduction}

Consider the $1$-dimensional reflected Brownian motion on $\mathbb{R}_{+}$.
For simplicity, we let
\begin{itemize}
\item $\phi_{r\to s}\left(z\right)$ be the generating function of the hitting
time of site $s$ starting from site $r$;
\item $\phi_{r\to s\vert\bcancel{t}}\left(z\right)$ be the generating
function of the hitting time of site $s$ starting from site $r$
without reaching site $t$. 
\end{itemize}
In the case of three consecutive sites $a$, $b$ and $c$ with $0<a<b<c$, 
the generating functions of the corresponding hitting times can be
found in \cite[p.~198 and p.~355]{handbook}: with $w=\sqrt{2z}$, 
\begin{align}
\phi_{a\rightarrow b}\left(z\right) & =\frac{\cosh\left(aw\right)}{\cosh\left(bw\right)}\label{eq:1dim1}\\
\phi_{b\to a\vert\bcancel{c}}\left(z\right) & =\frac{\sinh\left(\left(c-b\right)w\right)}{\sinh\left(\left(c-a\right)w\right)}\label{eq:1dim2}\\
\phi_{b\to c\vert\bcancel{a}}\left(z\right) & =\frac{\sinh\left(\left(b-a\right)w\right)}{\sinh\left(\left(c-a\right)w\right)}\label{eq:1dim3}
\end{align}
Let us now label $\left(N+1\right)$
sites on the real positive line as $0=a_{0}<a_{1}<a_{2}<\cdots<a_{N}$.

\subsection{The case of $3$ sites}

We consider here $3$ sites $0=a_{0}<a_{1}<a_{2}$; the study of
hitting times yields the following identity.
\begin{thm}
Let $0<a_{1}<a_{2}$, then we have, for any positive integer $n$,
the expansion of Euler polynomials as a convex combination of higher-order
Bernoulli polynomials as 
\begin{align}
&E_{n}\left(\frac{x}{2a_{2}}+\frac{3}{2}-2\frac{a_{1}}{a_{2}}\right)-E_{n}\left(\frac{x}{2a_{2}}+\frac{1}{2}\right)\nonumber\\
=&\left(n+1\right)\left(1-2\frac{a_{1}}{a_{2}}\right)\frac{2^{n}a_{1}^{n}}{a_{2}^{n}}\sum_{k=0}^{\infty}p_{k}B_{n}^{\left(k+1\right)}\left(\frac{x}{4a_{1}}+\frac{a_{2}}{4a_{1}}+\frac{k}{2}\right),\label{eq:EnBnk}
\end{align}
where the coefficients
\[
p_{k}=\frac{a_{1}}{a_{2}}\left(1-\frac{a_{1}}{a_{2}}\right)^{k}
\]
are the probability weights of a geometric distribution with parameter
$a_1/a_2$.
\end{thm}

\begin{proof}
From formulas above, we have
\[
\phi_{0\rightarrow a_{1}}\left(z\right)=\sech\left(a_{1}w\right),\ \phi_{0\rightarrow a_{2}}\left(z\right)=\sech\left(a_{2}w\right),\ \phi_{a_{1}\to a_{2}\vert\bcancel{0}}\left(z\right)=\frac{\sinh\left(a_{1}w\right)}{\sinh\left(a_{2}w\right)},
\]
and 
\[
\phi_{a_{1}\to0\vert\bcancel{a_{2}}}\left(z\right)=\frac{\sinh\left(\left(a_{2}-a_{1}\right)w\right)}{\sinh\left(a_{2}w\right)}.
\]
The process includes one possible loop between sites $0$ and $a_{1}$,
so that by \eqref{eq:OneLoop}, we have 
\begin{align}
\phi_{0\rightarrow a_{2}}\left(z\right) & =\phi_{0\rightarrow a_{1}}\left(z\right)\phi_{a_{1}\to a_{2}\vert\bcancel{0}}\left(z\right)\sum_{k=0}^{\infty}\left(\phi_{0\rightarrow a_{1}}\left(z\right)\phi_{a_{1}\to0\vert\bcancel{a_{2}}}\left(z\right)\right)^{k}\nonumber \\
 & =\sech\left(a_{1}w\right)\cdot\frac{\sinh\left(a_{1}w\right)}{\sinh\left(a_{2}w\right)}\sum_{k=0}^{\infty}\left(\sech\left(a_{1}w\right)\frac{\sinh\left(\left(a_{2}-a_{1}\right)w\right)}{\sinh\left(a_{2}w\right)}\right)^{k}\label{eq:1dim2N}\\
 & =\frac{\sinh\left(a_{1}w\right)}{\cosh\left(a_{1}w\right)\sinh\left(a_{2}w\right)}\cdot\frac{1}{1-\frac{\sinh\left(\left(a_{2}-a_{1}\right)w\right)}{\cosh\left(a_{1}w\right)\sinh\left(a_{2}w\right)}}\nonumber \\
 & =\frac{\sinh\left(a_{1}w\right)}{\cosh\left(a_{1}w\right)\sinh\left(a_{2}w\right)-\sinh\left(\left(a_{2}-a_{1}\right)w\right)},\nonumber 
\end{align}
which coincides with $\phi_{0\rightarrow a_{2}}\left(z\right)=\sech\left(a_{2}w\right)$,
since
\[
\sinh\left(\left(a_{2}-a_{1}\right)w\right)=\sinh\left(a_{2}w\right)\cosh\left(a_{1}w\right)-\cosh\left(a_{2}w\right)\sinh\left(a_{1}w\right).
\]
Meanwhile, we transform \eqref{eq:1dim2N} into a form involving Bernoulli,
uniform and Euler symbols as follows:
\begin{align*}
e^{a_{2}w\left(2\mathcal{E}+1\right)}= & \frac{a_{1}}{a_{2}}e^{a_{1}w\left(2\mathcal{U}-1\right)}e^{a_{2}w\left(2\mathcal{B}+1\right)}\\
 & \times\sum_{k=0}^{\infty}e^{a_{1}w\left(2\mathcal{E}^{\left(k+1\right)}+k+1\right)}\left(\frac{a_{2}-a_{1}}{a_{2}}\right)^{k}e^{\left(a_{2}-a_{1}\right)w\left(2\mathcal{U}^{\left(k\right)}-k\right)}e^{a_{2}w\left(2\mathcal{B}^{\left(k\right)}+k\right)}\\
= & e^{a_{1}w\left(2\mathcal{U}-1\right)+a_{2}w\left(2\mathcal{B}+1\right)}\\
 & \times\sum_{k=0}^{\infty}p_{k}e^{a_{1}w\left(2\mathcal{E}^{\left(k+1\right)}+k+1\right)+\left(a_{2}-a_{1}\right)w\left(2\mathcal{U}^{\left(k\right)}-k\right)+a_{2}w\left(2\mathcal{B}^{\left(k\right)}+k\right)}.
\end{align*}
Multiplying both sides by $e^{xw}$ and comparing coefficients of
$w^{n}$ yield
\begin{align*}
\left(x+2a_{2}\mathcal{E}+a_{2}\right)^{n}= & \sum_{k=0}^{\infty}p_{k}\bigg[x+a_{1}\left(2\mathcal{U}-1\right)+a_{2}\left(2\mathcal{B}+1\right)+a_{1}\left(2\mathcal{E}^{\left(k+1\right)}+k+1\right)\\
 & +\left(a_{2}-a_{1}\right)\left(2\mathcal{U}^{\left(k\right)}-k\right)+a_{2}\left(2\mathcal{B}^{\left(k\right)}+k\right)\bigg]^{n}\\
= & \sum_{k=0}^{\infty}p_{k}\left[x+2\left(a_{2}-a_{1}\right)\mathcal{B}+2a_{1}\mathcal{E}^{\left(k+1\right)}+2a_{1}\mathcal{B}^{\left(k\right)}+a_{2}+2a_{1}k\right]^{n},
\end{align*}
where, in the last step, we have used the fact that $\mathcal{B}$
and $\mathcal{U}$ cancel each other according to \eqref{eq:cancel}.
Now, the substitution $x\mapsto x+2\left(a_{2}-2a_{1}\right)\mathcal{U}$
yields 
\[
\left(x+2\left(a_{2}-2a_{1}\right)\mathcal{U}+2a_{2}\mathcal{E}+a_{2}\right)^{n}=\sum_{k=0}^{\infty}p_{k}\left(x+2a_{1}\mathcal{E}^{\left(k+1\right)}+2a_{1}\mathcal{B}^{\left(k+1\right)}+a_{2}+2a_{1}k\right)^{n}.
\]
Here, the left-hand side can be computed as
\begin{align*}
2^{n}a_{2}^{n}\left[\mathcal{E}+\frac{x}{2a_{2}}+\frac{1}{2}+\left(1-2\frac{a_{1}}{a_{2}}\right)\mathcal{U}\right]^{n} & =2^{n}a_{2}^{n}E_{n}\left(\frac{x}{2a_{2}}+\frac{1}{2}+\left(1-2\frac{a_{1}}{a_{2}}\right)\mathcal{U}\right)\\
 & =2^{n}a_{2}^{n}\int_{0}^{1}E_{n}\left(\frac{x}{2a_{2}}+\frac{1}{2}+\left(1-2\frac{a_{1}}{a_{2}}\right)t\right)\mathrm{d}t\\
 & =\frac{2^{n}a_{2}^{n}\left[E_{n}\left(\frac{x}{2a_{2}}+\frac{3}{2}-2\frac{a_{1}}{a_{2}}\right)-E_{n}\left(\frac{x}{2a_{2}}+\frac{1}{2}\right)\right]}{\left(1-2\frac{a_{1}}{a_{2}}\right)\left(n+1\right)};
\end{align*}
while the right-hand side is, by \eqref{eq:2BE}, 
\begin{align*}
 & \sum_{k=0}^{\infty}p_{k}\left(x+2a_{1}\mathcal{E}^{\left(k+1\right)}+2a_{1}\mathcal{B}^{\left(k+1\right)}+a_{2}+2a_{1}k\right)^{n}\\
= & \sum_{k=0}^{\infty}p_{k}\left(x+4a_{1}\mathcal{B}^{\left(k+1\right)}+a_{2}+2a_{1}k\right)^{n}\\
= & \sum_{k=0}^{\infty}p_{k}4^{n}a_{1}^{n}B_{n}^{\left(k+1\right)}\left(\frac{x}{4a_{1}}+\frac{a_{2}}{4a_{1}}+\frac{k}{2}\right).
\end{align*}
Further simplification completes the proof.
\end{proof}
\begin{rem}
For uniformly spaced levels $a_{1}=a_2/2=a/2$,
identity \eqref{eq:EnBnk} collapses to the trivial identity $0=0$. 
\end{rem}

\subsection{The case of $4$ sites}

We consider now the case of four sites $0=a_{0}<a_{1}<a_{2}<a_{3}$
with two possible loops, one between $a_{0}$ and $a_{1}$ and one
between $a_{1}$ and $a_{2}$. For the sake of simplicity, we consider
only the case where all sites uniformly distributed, with $a_{i}=i$
for $i=0,\ 1,\ 2,\ 3$, and obtain the following identity.
\begin{thm}
For any positive integer $n$, the Euler polynomial of degree $n$ can be expressed
as a linear combination of high-order Euler polynomials of the same degree as
\begin{equation}
E_{n}\left(x\right)=\sum_{k=0}^{\infty}\frac{3^{k-n}}{4^{k+1}}E_{n}^{\left(2k+3\right)}\left(3x+k\right).\label{eq:Thm3}
\end{equation}
\end{thm}

\begin{proof}
By \eqref{eq:TwoLoops} and the basic identity $\sinh\left(2w\right)=2\sinh\left(w\right)\cosh\left(w\right)$,
we have
\begin{align*}
\phi_{0\rightarrow3}\left(z\right)= & \phi_{0\rightarrow1}\left(z\right)\phi_{1\to2\vert\bcancel{0}}\left(z\right)\phi_{2\to3\vert\bcancel{1}}\left(z\right)\cdot\\
 & \times\sum_{k=0}^{\infty}\left(\phi_{0\to1}\left(z\right)\phi_{1\to0\vert\bcancel{2}}\left(z\right)+\phi_{1\to2\vert\bcancel{0}}\left(z\right)\phi_{2\to1\vert\bcancel{3}}\left(z\right)\right)^{k}\\
= & \sech\left(w\right)\left(\frac{\sinh\left(w\right)}{\sinh\left(2w\right)}\right)^{2}\sum_{k=0}^{\infty}\left(\sech\left(w\right)\frac{\sinh\left(w\right)}{\sinh\left(2w\right)}+\frac{\sinh\left(w\right)}{\sinh\left(2w\right)}\cdot\frac{\sinh\left(w\right)}{\sinh\left(2w\right)}\right)^{k}\\
= & \sum_{k=0}^{\infty}\frac{1}{4}\left(\frac{3}{4}\right)^{k}\sech^{2k+3}\left(w\right),
\end{align*}
which is equal to $\sech\left(3w\right)=\phi_{0\rightarrow3}\left(z\right)$.
Therefore, we deduce 
\[
e^{3w\left(2\mathcal{E}+1\right)}=\sum_{k=0}^{\infty}\frac{3^{k}}{4^{k+1}}e^{w\left(2\mathcal{E}^{\left(2k+3\right)}+2k+3\right)},
\]
namely, 
\[
e^{6w\mathcal{E}}=\sum_{k=0}^{\infty}\frac{3^{k}}{4^{k+1}}e^{2w\mathcal{E}^{\left(2k+3\right)}+2wk}.
\]
Multiplying both sides by $e^{xw}$ and comparing coefficients of
$w^{n}$, we obtain 
\[
\left(x+6\mathcal{E}\right)^{n}=\sum_{k=0}^{\infty}\frac{3^{k}}{4^{k+1}}\left(x+2k+2\mathcal{E}^{\left(2k+3\right)}\right)^{n},
\]
where the left-hand side is 
\[
\left(x+6\mathcal{E}\right)^{n}=6^{n}\left(\mathcal{E}+\frac{x}{6}\right)^{n}=6^{n}E_{n}\left(\frac{x}{6}\right);
\]
and the right-hand side is
\begin{align*}
\sum_{k=0}^{\infty}\frac{3^{k}}{4^{k+1}}\left(x+2k+2\mathcal{E}^{\left(2k+3\right)}\right)^{n} & =\sum_{k=0}^{\infty}\frac{3^{k}}{4^{k+1}}2^{n}\left(\frac{x}{2}+k+\mathcal{E}^{\left(2k+3\right)}\right)^{n}\\
 & =\sum_{k=0}^{\infty}\frac{3^{k}}{4^{k+1}}2^{n}E_{n}^{\left(2k+3\right)}\left(\frac{x}{2}+k\right).
\end{align*}
Simplification completes the proof. 
\end{proof}
In the general case of 4 sites that are not uniformly spaced, we obtain
a much more complicated result as follows:
\begin{thm}
For any $0<a_{1}<a_{2}<a_{3}$ and arbitrary positive integer $n$,  we have

\begin{align*}
E_{n}\left(\frac{x}{2a_{3}}+\frac{1}{2}\right)= & \sum_{k=0}^{\infty}\sum_{l=0}^{k}q_{k,l}\left(\frac{a_{1}}{2a_{3}}\right)^{n}E_{n}^{\left(l\right)}\left(\frac{x}{a_{1}}+2\frac{a_{2}-a_{1}}{a_{1}}\mathcal{B}+2\frac{a_{3}-a_{2}}{a_{1}}\mathcal{B}'+2\frac{a_{2}-a_{1}}{a_{1}}\mathcal{U}^{\left(l\right)}\right.\\
 & \left.+2\mathcal{U}'^{\left(k-l\right)}+2\frac{a_{2}-a_{1}}{a_{1}}\mathcal{B}'^{\left(k-l\right)}+\frac{r_{k,l}}{a_{1}}\right),
\end{align*}
with the coefficients
\[
q_{k,l}=\binom{k}{l}\frac{\left(a_{2}-a_{1}\right)^{l+1}a_{1}^{k-l+1}\left(a_{3}-a_{2}\right)^{k-l}}{a_{2}^{k+1}\left(a_{3}-a_{1}\right)^{k-l+1}},
\]
and 
\[
r_{k,l}=a_{3}+\left(2k-2l\right)a_{2}+\left(3l-k+1\right)a_{1}.
\]
\end{thm}

\begin{proof}
Apply \eqref{eq:TwoLoops} to obtain 
\begin{align*}
\phi_{0\rightarrow a_{3}}\left(z\right)= & \phi_{0\rightarrow a_{1}}\left(z\right)\phi_{a_{1}\to a_{2}\vert\bcancel{0}}\left(z\right)\phi_{a_{2}\to a_{3}\vert\bcancel{a_{1}}}\left(z\right)\\
 & \times\sum_{k=0}^{\infty}\left(\phi_{0\to a_{1}}\left(z\right)\phi_{a_{1}\to0\vert\bcancel{a_{2}}}\left(z\right)+\phi_{a_{1}\to a_{2}\vert\bcancel{0}}\left(z\right)\phi_{a_{2}\to a_{1}\vert\bcancel{a_{3}}}\left(z\right)\right)^{k}\\
= & \sech\left(a_{1}w\right)\cdot\frac{\sinh\left(a_{1}w\right)}{\sinh\left(a_{2}w\right)}\cdot\frac{\sinh\left(\left(a_{2}-a_{1}\right)w\right)}{\sinh\left(\left(a_{3}-a_{1}\right)w\right)}\\
 & \times\sum_{k=0}^{\infty}\left(\sech\left(a_{1}w\right)\cdot\frac{\sinh\left(\left(a_{2}-a_{1}\right)w\right)}{\sinh\left(a_{2}w\right)}+\frac{\sinh\left(a_{1}w\right)}{\sinh\left(a_{2}w\right)}\cdot\frac{\sinh\left(\left(a_{3}-a_{2}\right)w\right)}{\sinh\left(\left(a_{3}-a_{1}\right)w\right)}\right)^{k}\\
= & \sum_{k=0}^{\infty}\frac{1}{\sinh^{k+1}\left(a_{2}w\right)}\sum_{l=0}^{k}\binom{k}{l}\sech^{l+1}\left(a_{1}w\right)\sinh^{l+1}\left(\left(a_{2}-a_{1}\right)w\right)\\
 & \times\frac{\sinh^{k-l+1}\left(a_{1}w\right)\sinh^{k-l}\left(\left(a_{3}-a_{2}\right)w\right)}{\sinh^{k-l+1}\left(\left(a_{3}-a_{1}\right)w\right)}
\end{align*}
Now, we multiply both sides by $e^{xw}$ and compare coefficients
of $w^{n}$, with the notation

\[
q_{k,l}=\binom{k}{l}\frac{\left(a_{2}-a_{1}\right)^{l+1}a_{1}^{k-l+1}\left(a_{3}-a_{2}\right)^{k-l}}{a_{2}^{k+1}\left(a_{3}-a_{1}\right)^{k-l+1}}
\]
and 
\[
r_{k,l}=a_{3}+\left(2k-2l\right)a_{2}+\left(3l-k+1\right)a_{1},
\]
to obtain 
\begin{align*}
\left(x+2a_{3}\mathcal{E}+a_{3}\right)^{n}= & \sum_{k=0}^{\infty}\sum_{l=0}^{k}q_{k,l}\bigg[x+2\left(a_{2}-a_{1}\right)\mathcal{B}+2\left(a_{3}-a_{2}\right)\mathcal{B}'+a_{1}\mathcal{E}^{\left(l\right)}+2\left(a_{2}-a_{1}\right)\mathcal{U}^{\left(l\right)}\\
 & +2a_{1}\mathcal{U}'^{\left(k-l\right)}+2\left(a_{2}-a_{1}\right)\mathcal{B}'^{\left(k-l\right)}+r_{k,l}\bigg]^{n}.
\end{align*}
\end{proof}

\section{\label{sec:ThreeDim}Bessel Process in $\mathbb{R}^{3}$}

\subsection{Introduction}

We consider now a Bessel process $R_{t}^{\left(3\right)}$ in $\mathbb{R}^{3}$.
Using similar notations $\phi_{a\to b}\left(z\right)$ and $\phi_{a\to b\vert\bcancel{c}}\left(z\right)$
as in the previous section and considering sites $a<b<c$ that are
now concentric spheres of radii $a$, $b$, and $c$, we will need
the following formulas from \cite[pp.~463--464]{handbook}:

\begin{align}
\phi_{a\rightarrow b}\left(z\right) & =\frac{b\sinh\left(aw\right)}{a\sinh\left(bw\right)},\label{eq:3dim1}\\
\phi_{b\to c\vert\bcancel{a}}\left(z\right) & =\frac{c\sinh\left(\left(b-a\right)w\right)}{b\sinh\left(\left(c-a\right)w\right)},\label{eq:3dim2}\\
\phi_{b\to a\vert\bcancel{c}}\left(z\right) & =\frac{a\sinh\left(\left(c-b\right)w\right)}{c\sinh\left(\left(c-a\right)w\right)}.\label{eq:3dim3}
\end{align}
\begin{rem*}
One can easily check that 
\[
\phi_{a_{1}\to0\vert\bcancel{a_{2}}}\left(z\right)=\frac{0\cdot\sinh\left(\left(a_{2}-a_{1}\right)w\right)}{a_{2}\sinh\left(\left(a_{2}-0\right)w\right)}=0,
\]
so that the first possible loop is now between $a_{1}$ and $a_{2}$,
different from the $1$-dimensional case, where the first possible loop
was between the origin $0$ and the first site $a_{1}$.
\end{rem*}

\subsection{The case  $N=3$}
\begin{thm}
\label{thm:N3}In the case of three concentric spheres of arbitrary
radii $0<a_{1}<a_{2}<a_{3}$, we deduce the following expression of a Bernoulli polynomial of degree $n$ in terms of higher-order Bernoulli polynomials
\begin{eqnarray*}
B_{n}\left(\frac{x+a_{3}}{2a_{3}}\right) & = & \left(\frac{a_{2}}{a_{3}}\right)^{n}\sum_{k=0}^{\infty}p_{k}B_{n}^{\left(k+1\right)}\bigg[\frac{x}{2a_{2}}+\beta_{k}+\frac{a_{2}-a_{1}}{a_{2}}\mathcal{U}+\frac{a_{3}-a_{2}}{a_{2}}\mathcal{U}^{\left(k\right)}+\frac{a_{1}}{a_{2}}\mathcal{U}^{'\left(k\right)}\\
 &  & +\frac{a_{3}-a_{1}}{a_{2}}\mathcal{B}^{\left(k+1\right)}\bigg]^{n}
\end{eqnarray*}
with
\[
p_{k}=\left(1-\alpha\right)\alpha^{k},\thinspace\thinspace\beta_{k}=\frac{a_{3}+2ka_{2}-2ka_{1}}{2a_{3}}
\]
and 
\[
\alpha=\frac{\left(a_{3}-a_{2}\right)a_{1}}{\left(a_{3}-a_{1}\right)a_{2}}.
\]
\end{thm}

The special case of uniformly spaced radii is as follows.
\begin{cor}
\label{cor:2}In the case of three spheres of radii $a_{1}=1,a_{2}=2$, 
and $a_{3}=3$, we obtain the identity
\begin{equation}
\frac{3^{n+1}}{n+1}\left[B_{n+1}\left(\frac{x}{6}+\frac{5}{6}\right)-B_{n+1}\left(\frac{x}{6}+\frac{1}{2}\right)\right]=\frac{3}{4}\sum_{k=0}^{\infty}\left(\frac{1}{4}\right)^{k}E_{n}^{\left(2k+2\right)}\left(\frac{x+3+2k}{2}\right).\label{eq:123}
\end{equation}
\end{cor}

The specialization to the case $x=0$ and $n$ odd of 
this identity is
as follows.
\begin{cor}
The even Bernoulli number $B_{2m}$ can be expressed as a convex combination
of high-order Euler polynomials as follows
\[
B_{2m}=\frac{m}{\left(1-2^{1-2m}\right)\left(3^{2m}-1\right)}\sum_{k=0}^{\infty}\left(\frac{1}{4}\right)^{k}E_{2m-1}^{\left(2k+2\right)}\left(k+\frac{3}{2}\right).
\]
\end{cor}

\begin{proof}
Take $x=0$ and $n=2m-1$ in \eqref{eq:123}. By Entries 24.4.27 and
24.4.32 of \cite{NIST}, the left-hand side becomes
\begin{align*}
\frac{3^{2m}}{2m}\left[B_{2m}\left(\frac{5}{6}\right)-B_{2m}\left(\frac{1}{2}\right)\right] & =\frac{3^{2m}}{2m}\left[\frac{1}{2}\left(1-2^{1-2m}\right)\left(1-3^{1-2m}\right)B_{2m}+\left(1-2^{1-2m}\right)B_{2m}\right]\\
 & =\frac{3^{2m}}{2m}\left(1-2^{1-2m}\right)B_{2m}\left(\frac{1-3^{1-2m}}{2}+1\right)\\
 & =\frac{3}{4m}\left(1-2^{1-2m}\right)\left(3^{2m}-1\right)B_{2m};
\end{align*}
while the right-hand side is
\[
\sum_{k=0}^{\infty}\frac{3}{4}\left(\frac{1}{4}\right)^{k}E_{2m-1}^{\left(2k+2\right)}\left(k+\frac{3}{2}\right).
\]
And we obtain, after simplification, the desired result.
\end{proof}

\begin{proof}
[Proof of Theorem \ref{thm:N3}] Remarking that $\underset{x\rightarrow0}{\lim}\left[\sinh\left(xt\right)/x\right]=t$,
from \eqref{eq:3dim1}, \eqref{eq:3dim2} and \eqref{eq:3dim3}, we
have

\[
\begin{cases}
\phi_{a_{0}\rightarrow a_{1}}\left(z\right)=\frac{a_{1}w}{\sinh\left(a_{1}w\right)}, & \\
\phi_{a_{0}\rightarrow a_{3}}\left(z\right)=\frac{a_{3}w}{\sinh\left(a_{3}w\right)}, & \\
\phi_{a_{1}\to a_{2}\vert\bcancel{a_{0}}}\left(z\right)=\frac{a_{2}}{a_{1}}\frac{\sinh\left(a_{1}w\right)}{\sinh\left(a_{2}w\right)}, & \\
\phi_{a_{2}\to a_{1}\vert\bcancel{a_{3}}}\left(z\right)=\frac{a_{1}}{a_{2}}\frac{\sinh\left(\left(a_{3}-a_{2}\right)w\right)}{\sinh\left(\left(a_{3}-a_{1}\right)w\right)}, & \\
\phi_{a_{2}\to a_{3}\vert\bcancel{a_{1}}}\left(z\right)=\frac{a_{3}}{a_{2}}\frac{\sinh\left(\left(a_{2}-a_{1}\right)w\right)}{\sinh\left(\left(a_{3}-a_{1}\right)w\right)}, & 
\end{cases}
\]
so that
\begin{align*}
\phi_{a_{0}\rightarrow a_{3}}\left(z\right)= & \sum_{k=0}^{\infty}\phi_{a_{0}\rightarrow a_{1}}\left(z\right)\phi_{a_{1}\to a_{2}\vert\bcancel{a_{0}}}\left(z\right)\left(\phi_{a_{2}\to a_{1}\vert\bcancel{a_{3}}}\left(z\right)\phi_{a_{1}\to a_{2}\vert\bcancel{a_{0}}}\left(z\right)\right)^{k}\phi_{a_{2}\to a_{3}\vert\bcancel{a_{1}}}\left(z\right)\\
= & \left(\frac{a_{1}w}{\sinh\left(a_{1}w\right)}\right)\cdot\left(\frac{a_{2}\sinh\left(a_{1}w\right)}{a_{1}\sinh\left(a_{2}w\right)}\right)\left(\frac{a_{3}\sinh\left(\left(a_{2}-a_{1}\right)w\right)}{a_{2}\sinh\left(\left(a_{3}-a_{1}\right)w\right)}\right)\\
 & \times\sum_{k=0}^{\infty}\left(\frac{\sinh\left(\left(a_{3}-a_{2}\right)w\right)}{\sinh\left(\left(a_{3}-a_{1}\right)w\right)}\cdot\frac{\sinh\left(a_{1}w\right)}{\sinh\left(a_{2}w\right)}\right)^{k}\\
= & a_{3}w\sinh\left(\left(a_{2}-a_{1}\right)w\right)\sum_{k=0}^{\infty}\frac{\sinh^{k}\left(\left(a_{3}-a_{2}\right)w\right)\sinh^{k}\left(a_{1}w\right)}{\sinh^{k+1}\left(\left(a_{3}-a_{1}\right)w\right)\sinh^{k+1}\left(a_{2}w\right)}.
\end{align*}
Simple algebra shows that this expression coincides with 
\[
\phi_{a_{0}\rightarrow a_{3}}=\frac{a_{3}w}{\sinh\left(a_{3}w\right)}.
\]
Then, we obtain
\begin{align*}
e^{a_{3}w\left(2\mathcal{B}+1\right)}= & e^{w\left(a_{2}-a_{1}\right)\left(2\mathcal{U}-1\right)}\sum_{k=0}^{\infty}\bigg\{ r_{k}e^{w\left[\left(a_{3}-a_{2}\right)\left(2\mathcal{U}^{\left(k\right)}-k\right)+a_{1}\left(2\mathcal{U}'^{\left(k\right)}-k\right)\right]}\\
 & \times e^{w\left[\left(a_{3}-a_{1}\right)\left(2\mathcal{B}^{\left(k+1\right)}+k+1\right)+a_{2}\left(2\mathcal{B}'^{\left(k+1\right)}+k+1\right)\right]}\bigg\}
\end{align*}
where
\[
r_{k}=\frac{a_{3}\left(a_{2}-a_{1}\right)\left(a_{3}-a_{2}\right)^{k}a_{1}^{k}}{\left(a_{3}-a_{1}\right)^{k+1}a_{2}^{k+1}}.
\]
Similarly, we multiply both sides by $e^{wx}$ and look at 
the coefficients of $w^{n}$. For simplicity, we let
\[
s_{k}=a_{3}+2ka_{2}-2ka_{1},
\]
so that 
\begin{align*}
\left(2a_{3}\right)^{n}B_{n}\left(\frac{x+a_{3}}{2a_{3}}\right)= & \sum_{k=0}^{\infty}r_{k}\bigg[x+2a_{3}\mathcal{B}+a_{3}+s_{k}+2\left(a_{2}-a_{1}\right)\mathcal{U}+2\left(a_{3}-a_{2}\right)\mathcal{U}^{\left(k\right)}\\
 & +2a_{1}\mathcal{U}'^{\left(k\right)}+2\left(a_{3}-a_{1}\right)\mathcal{B}^{\left(k+1\right)}+2a_{2}\mathcal{B}'^{\left(k+1\right)}\bigg]^{n}.
\end{align*}
The special case $a_{i}=i$, with 
\[
\beta_{k}=s_{k}\bigg|_{a_{i}=i}=3+2k\ \ \ \text{and}\ \ \ \rho_{k}=r_{k}\bigg|_{a_{i}=i}=\frac{3}{4}\left(\frac{1}{4}\right)^{k},
\]
produces 
\begin{align*}
6^{n}B_{n}\left(\frac{x+3}{6}\right) & =\sum_{k\ge0}\rho_{k}\left(x+3+2k+2\mathcal{U}+2\mathcal{U}^{\left(k\right)}+2\mathcal{U}'^{\left(k\right)}+4\mathcal{B}^{\left(k+1\right)}+4\mathcal{B}'^{\left(k+1\right)}\right)^{n}\\
 & =\sum_{k\ge0}\rho_{k}\left(x+3+2k+2\mathcal{U}^{\left(2k+1\right)}+4\mathcal{B}^{\left(2k+2\right)}\right)^{n}\\
 & =\sum_{k\ge0}\rho_{k}\left(x+3+2k+2\mathcal{U}^{\left(2k+1\right)}+2\mathcal{B}^{\left(2k+2\right)}+2\mathcal{E}^{\left(2k+2\right)}\right)^{n}\\
 & =\sum_{k\ge0}\rho_{k}\left(x+3+2k+2\mathcal{B}+2\mathcal{E}^{\left(2k+2\right)}\right)^{n}\\
 & =\sum_{k\ge0}\rho_{k}2^{n}E_{n}^{\left(2k+2\right)}\left(\frac{x+3+2k}{2}+\mathcal{B}\right),
\end{align*}
or
\[
3^{n}B_{n}\left(\frac{x}{6}+\frac{1}{2}\right)=\sum_{k\ge0}\rho_{k}E_{n}^{\left(2k+2\right)}\left(\frac{x+3+2k}{2}+\mathcal{B}\right).
\]
Replacing $x$ by $x+2\mathcal{U}$  completes the proof.
\end{proof}

\subsection{The case $N=4$. }

We have now two possible loops among the sites with radii $a_{2}$, $a_{3}$ and $a_{4}$.
Following the notation in Subsection \ref{subsec:loops}, we have
\[
I_{a_{1},a_{2}}=\phi_{a_{1}\to a_{2}\vert\bcancel{a_{0}}}\left(z\right)\phi_{a_{2}\to a_{1}\vert\bcancel{a_{3}}}\left(z\right)=\frac{\sinh\left(a_{1}w\right)\sinh\left(\left(a_{3}-a_{2}\right)w\right)}{\sinh\left(a_{2}w\right)\sinh\left(\left(a_{3}-a_{1}\right)w\right)}
\]
and similarly, 
\[
I_{a_{2},a_{3}}=\phi_{a_{2}\to a_{3}\vert\bcancel{a_{1}}}\left(z\right)\phi_{a_{3}\to a_{2}\vert\bcancel{a_{4}}}\left(z\right)=\frac{\sinh\left(\left(a_{2}-a_{1}\right)w\right)\sinh\left(\left(a_{4}-a_{3}\right)w\right)}{\sinh\left(\left(a_{3}-a_{1}\right)w\right)\sinh\left(\left(a_{4}-a_{2}\right)w\right)}.
\]
In order to obtain a simple expression for $I_{a_{1},a_{2}}+I_{a_{2},a_{3}}$,
we further assume that the levels are uniformly distributed, i.e.,  $a_{i}=i$,
for $i=0,\ 1,\ 2,\ 3,\ 4$, so that
\[
I_{a_{1,}a_{2}}+I_{a_{2},a_{3}}=\frac{2\sinh^{2}\left(w\right)}{\sinh^{2}\left(2w\right)}=\frac{\sech^{2}\left(w\right)}{2}.
\]
This produces the following identity.
\begin{thm}
For any positive integer $n$, the Bernoulli polynomial of degree $n$ can be expanded
as the convex combination of higher-order Euler polynomials of the same degree as
\begin{equation}
B_{n}\left(\frac{x+4}{6}\right)=\frac{1}{3^{n}}\sum_{k=0}^{\infty}\frac{1}{2^{k}}E_{n}^{\left(2k+2\right)}\left(\frac{x+2k+3}{2}\right).\label{eq:Thm 7}
\end{equation}
\end{thm}

\begin{proof}
From
\[
\phi_{0\to4}\left(z\right)=\phi_{0\to1}\left(z\right)\phi_{1\to2\vert\bcancel{0}}\left(z\right)\phi_{2\to3\vert\bcancel{1}}\left(z\right)\phi_{3\to4\vert\bcancel{2}}\left(z\right)\sum_{k=0}^{\infty}\left(I_{a_1,a_2}+I_{a_2,a_3}\right)^{k},
\]
namely, 
\begin{align*}
\frac{4w}{\sinh\left(4w\right)} & =\frac{w}{\sinh\left(w\right)}\cdot\frac{2\sinh\left(w\right)}{\sinh\left(2w\right)}\cdot\frac{3\sinh\left(w\right)}{2\sinh\left(2w\right)}\cdot\frac{4\sinh\left(w\right)}{3\sinh\left(2w\right)}\sum_{k=0}^{\infty}\frac{\sech^{2k}\left(w\right)}{2^{k}}\\
 & =\frac{w}{\sinh\left(w\right)}\sum_{k=0}^{\infty}\frac{\sech^{2k+2}\left(w\right)}{2^{k}},
\end{align*}
we deduce that
\[
e^{4w\left(2\mathcal{B}+1\right)}=e^{w\left(2\mathcal{B}'+1\right)}\sum_{k=0}^{\infty}\frac{1}{2^{k}}e^{w\left(2\mathcal{E}^{\left(2k+2\right)}+2k+2\right)}.
\]
After multiplying by $e^{xw}$, coefficients of $w^{n}$ on both sides
yield
\[
\left(x+8\mathcal{B}+4\right)^{n}=\sum_{k=0}^{\infty}\frac{\left[x+2\mathcal{B}'+1+2\mathcal{E}^{\left(2k+2\right)}+2k+2\right]^{n}}{2^{k}}.
\]
Apply the substitution $x\mapsto x+2\mathcal{U}$, to get the left-hand
side as
\[
\left(x+6\mathcal{B}+4\right)^{n}=6^{n}\left(\frac{x+4}{6}+\mathcal{B}\right)^{n}=6^{n}B_{n}\left(\frac{x+4}{6}\right),
\]
and the right-hand side as
\begin{align*}
\sum_{k=0}^{\infty}\frac{\left[x+2\mathcal{E}^{\left(2k+2\right)}+2k+3\right]^{n}}{2^{k}} & =2^{n}\sum_{k=0}^{\infty}\frac{\left[\mathcal{E}^{\left(2k+2\right)}+\frac{x+2k+3}{2}\right]^{n}}{2^{k}}\\
 & =2^{n}\sum_{k=0}^{\infty}\frac{E_{n}^{\left(2k+2\right)}\left(\frac{x+2k+3}{2}\right)}{2^{k}}.
\end{align*}
Further simplification completes the proof.
\end{proof}
\begin{rem}
Note the analogy between \eqref{eq:Thm 7} and \eqref{eq:Thm3} whereas
they are obtained from two different setups.
\end{rem}

\section{Conclusion}

We have shown how the setup of random processes allows to obtain non trivial identities
among higher-order Bernoulli and Euler polynomials. The underlying
principle of this approach is that these special functions appear
naturally in the generating functions of the hitting times of these random
processes. 

Several remarks are in order at this point:

- the identities obtained from this approach are not of the usual,
convolutional type - see for example \cite{Dilcher} for such identities. They are rather connection-type identities between the usual Bernoulli and Euler polynomials and their higher-order counterparts;

- these inherently involve a mixture of higher-order Bernoulli and
Euler polynomials;

- the interest of this approach is that each term in the obtained decomposition can be physically related to a relevant object, namely one loop in a trajectory of a random process

- this work should be considered as a first approach only to a more
general project in which the richness of the possible setups for random walks
is expected to generate a number of non-trivial identities about more
general special functions.

\section*{Acknowledgement}

The authors would like to thank Prof.~P.~Salminen and Prof.~B.~Q.~Ta
for their valuable suggestions on this work. They also thank Prof.~K.~Dilcher for the organization of the Eighteenth International Conference on Fibonacci Numbers and their Applications.

\end{document}